\documentclass{article} \usepackage[latin1]{inputenc}
\usepackage{amsmath,enumitem,xfrac,fix-cm,color}  
\usepackage{amsthm}
\usepackage{amsfonts,amssymb,mathtools,stmaryrd,dsfont}
\usepackage[cm]{fullpage}
\numberwithin{equation}{section}
\newtheorem{theorem}{Theorem}[section]

\makeatletter\let\c@example\c@theorem\let\c@example\c@theorem\makeatother
\makeatletter\let\c@definition\c@theorem\let\c@definition\c@theorem\makeatother
\makeatletter\let\c@lemma\c@theorem\let\c@lemma\c@theorem\makeatother

\newcommand{\Partial}[2]{\frac{\partial #1}{\partial #2}}
\newcommand{\F}[1]{\mathds{ #1 }}\newcommand{\C}[1]{\mathcal{ #1 }}
\newcommand{\Vector}[3]{ \begin{pmatrix} #1 \\ #2 \\ #3 \end{pmatrix}}
\newcommand{\MTwo}[4]{\begin{pmatrix} #1 & #2 \\ #3 & #4 \end{pmatrix}}
\newcommand{\norm}[1]{\left\lVert #1 \right\rVert}\renewcommand{\sfrac}[2]{^{#1}\!\!/\!_{#2}}
\newcommand{\IP}[2]{\left\langle #1,#2 \right\rangle}
\newcommand{\op}[1]{\operatorname{#1}}\newcommand{\overtext}[2]{\stackrel{\text{#1}}{#2}}
\newcommand{\splitln}[4]{\left\{\begin{array}{cc} #1 & #2 \\ #3 & #4\end{array}\right.}
\newcommand{\Rn}{\F R^n}\renewcommand{\P}{\F{P}}\newcommand{\E}{\F{E}}\newcommand{\1}{\F{1}}
\newcommand{\tab}{\indent\indent}\newcommand{\Tab}{\newline\tab}
\renewcommand{\circle}[1]{\raisebox{.5pt}{\rm{\textcircled{\raisebox{-.9pt} {#1}}}}}
\newcommand{\intbar}{-\hspace{-10.5pt}\int}\newcommand{\todo}[1]{{\color{red}#1}}
\DeclareMathOperator{\st}{\;s.t.\;}\DeclareMathOperator{\as}{\;a.s.\;}\renewcommand{\epsilon}{\varepsilon}
\DeclareMathOperator*{\Union}{\bigcup}\DeclareMathOperator{\union}{\cup}
\DeclareMathOperator*{\Intersect}{\bigcap}\DeclareMathOperator{\intersect}{\cap}
\renewcommand{\bar}{\overline}\renewcommand{\hat}{\widehat}\renewcommand{\tilde}{\widetilde}
\def\lesssim{\stackrel{<}\sim}
\def\gtrsim{\stackrel{>}\sim}
\def\a{{\bf a}}
\def\c{{\bf c}}
\def\w{{\bf w}}
\def\x{{\bf x}}
\def\x'{{\bf x'}}
\def\x''{{\bf x''}}
\def\u{{\bf u}}
\def\v{{\bf v}}
\def\y{{\bf y}}
\def\z{{\bf z}}
\def\t{{\bf t}}
\def\e{{\bf e}}
\def\1{{\bf 1}}
\def\0{{\bf 0}}
\def\ie{{i.e.}}
\def\Ie{{I.e.}}
\def\RR{{\mathbb R}}
\def\RRk{{\RR^k}}
\def\l{{\ell}}
\newcommand{\bea}{\begin{eqnarray}}
\newcommand{\eea}{\end{eqnarray}}
\newcommand{\bean}{\begin{eqnarray*}}
\newcommand{\eean}{\end{eqnarray*}}
\newcommand{\beq}{\begin{equation}}
\newcommand{\eeq}{\end{equation}}
\newcommand{\bac}{\begin{array}{c}}
\newcommand{\ball}{\begin{array}{ll}}
\newcommand{\ea}{\end{array}}

\newcommand{\bbB}{{\mathbb B}}
\newcommand{\bbR}{{\mathbb R}}

\newcommand{\aA}{{\mathcal A}}
\newcommand{\EE}{{\mathcal E}}
\newcommand{\HH}{{\mathcal H}}
\newcommand{\II}{{\mathcal I}}
\newcommand{\OO}{{\mathcal O}}

\def\ole{\preceq}
\def\oge{\succeq}
\def\wh{\widehat}
\def\wt{\widetilde}

\def\al{\alpha}
\def\bt{\beta}
\def\ga{\gamma}
\def\Ga{\Gamma}
\def\bGa{{\mbox{\bm$\Ga$}}}
\def\de{\delta}
\def\bde{{\mbox{\bm$\de$}}}
\def\bga{{\mbox{\bm$\ga$}}}
\def\De{\Delta}
\def\ep{\varepsilon}
\def\ka{\kappa}
\def\la{\lambda}
\def\La{\Lambda}
\def\om{\omega}
\def\bom{{\mbox{\bm$\om$}}}
\def\sbom{{\mbox{\scriptsize\bm$\om$}}}
\def\Om{\Omega}
\def\si{\sigma}
\def\ta{\tau}
\def\th{\theta}
\def\vp{\varphi}
\def\bvp{{\mbox{\bm$\varphi$}}}
\def\ze{\zeta}
\def\tze{\widetilde{\zeta}}
\def\tz{\widetilde{z}}

\def\dia{{\rm diam}}
\def\dst{{\rm dist}}
\def\ov{\overline}
\def\rest{\Biggm|}
\def\intl{\int\limits}
\def\tosta{\stackrel{\ast}{\longrightarrow}}
\def\extom{{\mbox{{\bm$\om$}$_s^{\ast}(A,N)$}}}
\def\sextom{{\mbox{{\scriptsize\bm$\om$}{\scriptsize $_s^{\ast}(A,N)$}}}}
\def\UnuN{U^{\nu_{s,N}^{\ast}}}
\def\Rd{\bbR^d}
\def\Rdp{\bbR^{d^\prime}}
\def\dbl#1#2{\substack{#1\\#2}}
\def\sumj{\sideset{}{^{(j)}}\sum}
\def\sumjo{\sideset{}{_1^{(j)}}\sum}
\def\sumjt{\sideset{}{_2^{(j)}}\sum}
\newcommand{\fin}{\hfill\raisebox{-0.5ex}{\mbox{$\,\rule{2mm}{2mm}\,$}}\medskip}
\numberwithin{equation}{section}
\numberwithin{theorem}{section}

\def\({\left(} 
 
\def\){\right)}
\def\lb{\left[}
\def\rb{\right]}
\def\bl{\left\{}
\def\br{\right\}}
\def\mr{\right|}
\def\nl{\left\|}
\def\nr{\right\|}

\author{}
\title{ }

\def\ml{\left|}
\begin{document}

\title{On best uniform approximation of finite sets by linear combinations of real valued functions using linear programming.}

\author{Steven B. Damelin \thanks{Department of Mathematics, University of Michigan, 530 Church Street, Ann Arbor, MI 48109, USA: damelin@umich.eduj} and Michael Werman\thanks{Department of Computer Science, The Hebrew University, 91904, Jerusalem, Israel: michael.werman@mail.huji.ac.il}.}

\maketitle
 
\begin{abstract}
In this paper, we study the problem of best uniform approximation of finite sets by linear combinations of real valued functions using linear programming. Our study concerns the analysis of the best approximation problem:
\[
\displaystyle
\min_{\alpha\in \mathbb R^m}\max_{1\leq i\leq n}\left|y_i -\sum_{j=1}^m \alpha_j \Gamma_j ({\bf x}_i)
\right|.
\]
Here: $\Gamma:=\left\{\Gamma_1,...,\Gamma_m\right\}$ is a list of functions where for each $1\leq j\leq m$, $\Gamma_j:\Delta \rightarrow \mathbb R$ with 
$\Delta$ a set of evaluation points 
$\left\{{\bf x_1},...,{\bf x_n}\right\}$. $\left\{y_1,...,y_n\right\}$ is a set of real values and 
$\alpha:=(\alpha_1,...,\alpha_m)$ with $\alpha_j\in \mathbb R,\, 1\leq j\leq m$. 
\end{abstract}

\section{Introduction}
Classical approximation theory is concerned with how functions can best be approximated by simpler functions for example in characterizing errors of approximation in some well defined sense. Note that what is meant by best and simpler will depend on the application, metrics and spaces over which the approximation takes place. See for example \cite{DRW,DE,De,De1,M,L,L1} and the many references therein for an interesting qualitative and quantitative perspective on this beautiful subject. To gain some perspective, a classical result in approximation theory is Jackson's theorem which tells us that the error of best uniform approximation
of functions $f:(a,b)\to \mathbb R$ with $r\geq 1$ derivatives by the space of polynomials $\Pi_n$ of degree at most $n\geq 1$
is of order $\frac{1}{n^r}$. Here, $-\infty<a<b<\infty$. More precisely, for fixed $a,b,r,f$ as above,
\[
{\rm min}_{p\in \Pi_n}{\rm max}_{x\in (a,b)}|f(x)-p(x)|\leq \frac{C}{n^r}
\]
where $C$ is a positive constant depending on $f,a,b,r$.  
\medskip

Results of this kind represent popular areas of research for different spaces of functions over different domains approximated by elements of numerous linear and non-linear spaces with different notions of error measure. 
See for example \cite{DRW,DE,De,De1,M,L,L1} and the many references cited therein. On the other hand, linear programming \cite{V89,dual,lp} is a popular technique for the optimization of a linear objective function, subject to linear equality and linear inequality constraints.
The idea of this paper is to apply linear programming to the classical approximation problem of best uniform approximation. More specifically, we analyze the problem of  best uniform approximation of finite sets by linear combinations of real valued functions using linear programming. 
This amounts to  using linear programming, to analyze in various ways the best approximation problem of points rather than functions. 
The sets over which we work can be quite arbitrary. Henceforth, the same symbol for function, set, evaluation, and value may denote the same or different symbol at different times. The context will be clear.
\medskip

To make our idea more precise, we begin with:

\section{Our Setup}
Let $\Gamma:=\left\{\Gamma_1,...,\Gamma_m\right\}$ be a list of functions where for each $1\leq j\leq m$, $\Gamma_j:\Delta \rightarrow {\mathbb R}$ with $\Delta$ a set of evaluation points 
$\left\{{\bf x_1},...,{\bf x_n}\right\}$. $\left\{y_1,...,y_n\right\}$ is a set of real values. 
\medskip

Affine functions $(1,x,y,\dots,z)$, polynomials: $(1,x,x^2,x^3...)$, $(1,x,y,\dots,x^2,xy,\dots)$ and the random list $(1,xy^2,e^x,cos(y-x)\dots)$ are all examples of possible lists $\Gamma$. Note that the sets $\Delta$ are arbitrary.
\medskip

With $\alpha:=(\alpha_1,...,\alpha_m)$, $\alpha_j\in \mathbb R,\, 1\leq j\leq m$, we consider now the following best uniform approximation problem: 

\begin{equation}
\displaystyle
\min_{\alpha\in \mathbb R^m}\max_{1\leq i\leq n}\left|y_i -\sum_{j=1}^m \alpha_j \Gamma_j ({\bf x}_i)
\right|.
\end{equation}
\bigskip
$(2.1)$ is a convex optimization problem in the coefficients $\alpha$. 
\medskip
Indeed, checking easily, we have that for $0\le t \le 1$ 
$
\\
\left|y_i -\sum_{j=1}^m (t\alpha_j+(1-t)\beta_j) \Gamma_j ({\bf x_i})
\right|_\Delta\le
\left|ty_i -\sum_{j=1}^m t\alpha_j \Gamma_j ({\bf x_i})
\right|_\Delta+\left|(1-t)y_i -\sum_{j=1}^m (1-t)\beta_j \Gamma_j ({\bf x_i})
\right|_\Delta
$\\
$=t\left|y_i-\sum_{j=1}^m \alpha_j \Gamma_j ({\bf x_i})
\right|_\Delta+(1-t)\left|y_i -\sum_{j=1}^m \beta_j \Gamma_j ({\bf x_i})
\right|_\Delta\le d$. 
\bigskip
whenever, 
\[
\left|y_i -\sum_{j=1}^m \alpha_j \Gamma_j ({\bf x_i})
\right|_\Delta\le d
\] and 
\[
\left|y_i -\sum_{j=1}^m \beta_j \Gamma_j ({\bf x_i})
\right|_\Delta\le d.
\]
Here, for ease of notation, we have written 
$\max_{1\leq i\leq n}\left|f({\bf x_i}\right|$ as $|f({\bf x})|_\Delta$.
\bigskip

The objective of this paper is to study $(2.1)$ in various ways.   
\medskip

\section{Analysis of $(2.1)$ via linear programming}
In this section, we begin our analysis of $(2.1)$ via linear programming. Consider the linear program with the $m+1$ variables $\alpha_j$, $z\in \mathbb R$
and $2n$ constraints:

\begin{equation*}
\setlength\arraycolsep{1.5pt}
  \begin{array}{l@{\quad} r c r c r}
    \min          & z    \\
    \mathrm{s.t.} &  z & \ge &  \sum_{j=1}^m \alpha_j \Gamma_j ({\bf x}_1)-y_1\\
                  &   z & \ge & y_1 -\sum_{j=1}^m \alpha_j \Gamma_j ({\bf x}_1) \\
                  &  & \vdots\\
                 &  z& \ge &  \sum_{j=1}^m \alpha_j \Gamma_j ({\bf x}_n)-y_n \\
                  &    z & \ge & y_n -\sum_{j=1}^m \alpha_j \Gamma_j ({\bf x}_n) \\
  \end{array}
\end{equation*}
\medskip

If we write the above as the linear program

\begin{equation*}
\setlength\arraycolsep{1.5pt}
  \begin{array}{l@{\quad} r c r c r}
 \min z \\ \mathrm{s.t.}\\  &
\begin{pmatrix}
    \Gamma_1(x_1) & \Gamma_2(x_1) &\cdots& \Gamma_m(x_1)&-1\\
   - \Gamma_1(x_1) & -\Gamma_2(x_1) &\cdots& -\Gamma_m(x_1)&-1\\&&\vdots\\
    -\Gamma_1(x_n) & -\Gamma_2(x_n) &\cdots& -\Gamma_m(x_n)&-1
\end{pmatrix}
\begin{pmatrix}
    \alpha_1\\ \alpha_2 \\ \vdots \\ z
\end{pmatrix}
\le 
\begin{pmatrix}
    y_1\\ -y_1 \\ \vdots \\ -y_n
\end{pmatrix}
\end{array}
\end{equation*}
we obtain our first theorem.
\begin{theorem}
There  are $m+1$ ${\bf x}$'s all with the same  maximum discrepancy (discrepancy being the optimal $z$ solution.)  
\label{t:theorem1}
\end{theorem}

\begin{proof}
We look at each pair of odd and even adjacent rows, odd rows are where the interpolant is less than the value and even where they are above. Given we have $m+1$ columns and $2n$ rows there can be at most one of the pair $2i-1$, $2i$ of rows with non zero $\alpha_j$ an optimal solution, unless there is an interpolation with $0$ error. At an  optimal solution the inequalities are equality's and
this gives the $m+1$ sought after ${\bf x}'s$, all with the same maximum discrepancy.
\end{proof}

Our dual linear program is readily seen to be

\begin{equation*}
\setlength\arraycolsep{1.5pt}
  \begin{array}{l@{\quad} r c r c r}
    \min (y_1, -y_1, \cdots, -y_{n})(\beta_1,\dots,\beta_{2n}) \\\mathrm{s.t.}\\  
\begin{pmatrix}
    \Gamma_1(x_1) & -\Gamma_1(x_1) &\cdots& -\Gamma_1(x_n)\\
     \Gamma_2(x_1) & -\Gamma_2(x_1) &\cdots& -\Gamma_2(x_n)\\ &&\vdots \\
    \Gamma_m(x_1) & -\Gamma_m(x_1) &\cdots& -\Gamma_m(x_n)\\
    1&1&\cdots&1
\end{pmatrix}
\begin{pmatrix}
    \beta_1\\ \beta_2 \\ \vdots \\ \beta_{2n}
\end{pmatrix}
=
\begin{pmatrix}
    0\\ 0 \\ \vdots \\ 0\\ 1
\end{pmatrix}\\
\beta_i\ge 0.
\end{array}
\end{equation*}

With this in mind, let $d$ be  minimal discrepency where  $\Gamma=\sum \a_i \Gamma_i$ is
the optimal interpolator and $b_i$ optimal nonzero $\beta_i$, which are non zero iff row $i$ of the primal constraints is an equality. Matrix row and column manipulation give: 

\begin{itemize}
    \item[a] $\sum\limits_{i~odd}b_i y_\frac{i+1}{2}-\sum\limits_{j~ even}b_j y_\frac{j}{2}=d$, by strong duality.
    \item[b] ${i~odd}\rightarrow \Gamma({\bf x_\frac{i+1}{2}})-y_\frac{i+1}{2}=d$, given that for an optimal solution inequality becomes equality.
    \item[c] ${i~even}\rightarrow y_{\frac{i}{2}}-\Gamma({\bf x_\frac{i}{2}})=d$, given that for an optimal solution inequality becomes equality.
    \item[d] $\sum b_i=1$, from the last row of the dual matrix,
    \item[e] $\forall k, ~~ \sum\limits_{i~odd}b_i \Gamma_k ({\bf x_\frac{i+1}{2}})-\sum\limits_{j~ even}b_j \Gamma_k ({\bf x_\frac{j}{2}})=0$, from a row of the dual matrix,
    \item[f] 
    $\forall k, ~~ a_k\sum\limits_{i~odd}b_i \Gamma_k ({\bf x_\frac{i+1}{2}})-a_k\sum\limits_{j~ even}b_j \Gamma_k (x_\frac{j}{2})=0$, from a row of the primal matrix.
    \item[g] $\sum\limits_{i~odd}b_i \Gamma ({\bf x_\frac{i+1}{2}})-\sum\limits_{j~ even}b_j \Gamma ({\bf x_\frac{j}{2}})=0$.
    \item[h] $\sum\limits_{i~odd}b_i (\Gamma ({\bf x_\frac{i+1}{2}})- y_\frac{i+1}{2})+\sum\limits_{j~ even}b_j (y_\frac{j}{2}-\Gamma ({\bf x_\frac{j}{2}}))=d$.
     \end{itemize}
\medskip

We then have our second theorem given by:
\medskip
\begin{theorem}
If  $\Gamma_1$ is the constant $\mathds{1}$, then there is at least one overshoot and one undershoot with optimal discrepancy.
\label{t:theorem2}
\end{theorem}
\begin{proof}
Adding and subtracting rows 1 and 2n in our dual matrix
gives 
\[
\beta_1+\beta_3+\beta_5\cdots+\beta_{2n-1} =\frac{1}{2} \]
\[
\beta_2+\beta_4+\beta_6 \cdots +\beta_{2n}=\frac{1}{2}\]
which gives at least one overshoot and at least one undershoot. Even rows 
are above and odd rows below. 
Geometrically, we can use $\mathds{1}$ to move the solution up or down so that it is equidistant from extremal points. 
\end{proof}
We remark that we do not handle the case when $\Gamma$ and $\Delta$ produce a low rank matrix in the primal, that is, when there is an exact interpolation. 

\subsection{Weighted approximation and weighted version of (2.1)}

In classical approximation theory, errors of weighted approximation are often studied for different classes of real valued functions over domains in $\mathbb R^m$ where the function grows without bound for a large argument.
If one aims to speak to errors of best approximation of such functions, one should typically work with weighted metrics where the weight dampens the behavior of the function for a large argument. A typical example of such a scenario would be the following. See for example \cite{L,L1}. Suppose we are given continuous functions 
$f:\mathbb R\to \mathbb R$. Suppose that $f$ grows without bound for large argument. Then with even weights, $\mu:\mathbb R\to (0,\infty)$ of suitable smoothness and fast enough decrease for large argument with $|f\mu|$ having limit 0 for large argument
\[
{\rm inf}_{p\in \Pi_n}{\rm sup}_{x\in \mathbb R}|(f(x)-p(x))\mu(x)|
\]
has limit $0$ as $n$ increases without bound and  if $f$ is smooth enough
\[
{\rm inf}_{p\in \Pi_n}{\rm sup}_{x\in \mathbb R}|(f(x)-p(x))\mu(x)|\leq Cg(n)
\]
for an explicit sequence $g(n)$ depending on $\mu$ and decreasing to $0$ for large $n$. 
$C$ is a positive constant depending on $f$.
Similar problems make sense and are studied for non-even $\mu$ and when $f$ is real valued over a compact domain in $\mathbb R^m$ and has singularities on the boundary.

An important observation regarding $(2.1)$ is that it holds for a fixed weighted setting. That is $(2.1)$ has a natural weighted analog for weighted best uniform approximation of finite sets by weighted linear combinations of real valued functions. 
\medskip
Indeed we have:
\medskip
\begin{theorem}
Suppose we are given a non-negative set of weights $[\mu_1,...,\mu_n]$.
\medskip
Consider the best approximation problem: 

\beq
\displaystyle
\min_{\alpha\in \mathbb R^m} \max_{1\leq i\leq n}\left( \mu_i\left|{\bf y_i} -\sum_{j=1}^m \alpha_j \Gamma_j (s)\right|
\right)
\eeq

Then the analysis of $(3.1)$ including natural analogues of Theorem~\ref{t:theorem1} and Theorem~\ref{t:theorem2} follow in the same way as for $(2.1)$. 
\end{theorem}
\begin{proof}
Note that as in $(2.1)$, $(3.1)$ is a well defined convex optimization problem in the coefficients 
$\alpha$. It suffices to follow carefully the analysis for $(2.1)$ above and re scale.
\footnote{Simply re scale: $\mu_i\Gamma_j(.)$ and $\mu_iy_i$}.
\end{proof}

Following the work of Vaidya, \cite{V89}, our algorithms run in $O((n+m)^{1.5}mL)$ arithmetic operations in the worst case. Here $n$ is the number of constraints, $m$ is the number of variables, and $L$ is the number of bits. There are many packages \cite{lp} that can handle millions of $m,n$. 
\bigskip

{\bf Remark}\, An important idea in best approximation is that of equioscillation. 
The classical Chebyshev equioscillation theorem, see for example \cite{DRW,De1,M,L} is the following:
\bigskip

Let $f:[a,b]\to \mathbb R$, $-\infty<a<b<\infty$ be continuous.
Then a polynomial $p*\in \Pi_n$ is a best approximation to $f$ satisfying 
\[
{\rm min}_{p\in \Pi_n}\max_{x\in [a,b]}|f(x)-p(x)|=\max_{x\in [a,b]}|f(x)-p*(x)|
\]
exists if there exist $n+2$ points $\left\{x_1,..., x_{n+2}\right\}$
with $a\leq x_1<...<x_{n+2}\leq b$ such that 
\beq
f(x_i)-p*(x_i)=w(-1)^i{\max}_{[a,b]}|f(x)-p*(x)|,\, w=\pm 1.
\label{t:equioc}
\eeq
In this remark, we show how the the ideas in Theorem 3.1 and Theorem 3.3 can be used to constructing a case of failure of equioscillation and to establishing the known fact that polynomials equioscillate. 
\bigskip
\begin{itemize}
\item (A) {\bf A case of failure of equioscillation}; We argue as follows.
\medskip
If the optimal discrepancy  for $\Gamma,{\bf x},y$ is achieved at a ${\bf x_j}$ where
$\sum_{j=1}^m \alpha_j \Gamma_j ({\bf x_j})-{\bf y_i}=d$
 we define 
  \[
\begin{array}{llll}
     &\Gamma_k'({\bf x}) &=&  \begin{array} {cc}
     -\Gamma_k& {\bf x}={\bf x_j}\\
     \Gamma_k&otherwise
\end{array}\\
&     y'&=& \begin{array} {cc}  -y_j &j\\
     y &otherwise
     \end{array}\\
\end{array}\\
\]
Then substituting we find that for $\Gamma',{\bf x},y'$
the optimal solution, as it is for the same linear program (not the same data fitting problem), is at 
 $y'_j -\sum_{j=1}^m \alpha_j \Gamma'_j ({\bf x_j})=d$. This proves the required statement, as we can force the interpolator to be on only one side of the data, by changing the interpolating functions and data.
\item(B) {\bf Polynomials equioscillate}: We argue as follows. 
Let $t$ be the degree of the polynomial, that is the sum of $t+1$ monomials. Then at an optimal point   there are $t+2$ 
${\bf x}'s$ with maximal discrepancy $d$, let them be ${ z}_1\cdots {z}_{t+2}$, in increasing order, and the their respective values ($y$'s) be $q$'s. 
If our claim is not true 
then there are   ${z_j}$ and 
${z_{j+1}}$ so that there respective $q$'s are on the same side of the interpolant, (WLOG), $q_j-f({ z_j})=q_{j+1}-f({ z_{j+1}})=d$. 
We now choose a small  $\varepsilon>0$.  Then 
\begin{itemize}
    \item There is  a unique $m$-degree Lagrange polynomial interpolant $f'$ interpolating  the $t+1$ points, \\$({z_i},f(z_i))_{i\ne j,j+1} \union (z_j,f(z_j)+\varepsilon)$.
    \item $f'({z_{j+1}})\ge f({z_{j+1}})$ as
    $f'({z_{j+1}})- f({z_{j+1}}) =$
$\epsilon \prod\limits_{i=1,\, i\neq j,j+1}^{t+2} \frac{{z_{j+1}}-{z_i}}{{z_j}-{z_i}}$ 
which is positive as there is no ${z_i}$  between ${z_j}$ and ${z_{j+1}}$.
\item $f'$ interpolates all the points with discrepancy less or equal to $d$ and is exactly $d$ at $t$ points.
\end{itemize}

Let  $\delta >0$ and $f''$ be the Lagrange polynomial interpolating the $t+1 $ points  $(z_i,f'(z_i)+\delta*\text{sgn}(q_i-f'(z_i))$  where $i\neq j+1$. 
We are done as all points now have discrepancy less than $d$.
\end{itemize}

Similar but highly non trivial extensions of (A-B) hold in the weighted case. These are currently being studied by us together with replacing $\mathbb R^m$ by a more general metric space.

\end{document}